\documentclass[12pt]{article}
\usepackage[latin1]{inputenc}
\usepackage[british]{babel}
\usepackage{lmodern}
\usepackage[T1]{fontenc}
\usepackage[paper=a4paper, left=28mm, right=25mm, top=29mm, bottom=25mm]{geometry}
\usepackage{latexsym,amsfonts,amsmath,graphics}
\usepackage{epsfig}
\usepackage{enumitem}
\usepackage{amssymb,mathrsfs}
\usepackage{verbatim}
\newtheorem{theorem}{Theorem}
\newtheorem{lemma}{Lemma}
\newtheorem{corollary}{Corollary}
\newtheorem{proposition}{Proposition}

\newtheorem{remark}{Remark}

\newenvironment{proof}{\begin{trivlist} \item[\hskip\labelsep{\it Proof.}]}{$\hfill\Box$\end{trivlist}}

\newcommand{\rd}{\,\mathrm{d}}

\newcommand{\bsj}{\boldsymbol{j}}
\newcommand{\bsx}{\boldsymbol{x}}

\newcommand{\bsz}{\boldsymbol{z}}

\newcommand{\bsk}{\boldsymbol{k}}
\newcommand{\bsm}{\boldsymbol{m}}
\newcommand{\bst}{\boldsymbol{t}}

\newcommand{\bszero}{\boldsymbol{0}}
\newcommand{\bsone}{\boldsymbol{1}}

\newcommand{\RR}{\mathbb{R}}

\newcommand{\NN}{\mathbb{N}}

\newcommand{\bbD}{\mathbb{D}}

\newcommand{\calD}{\mathcal{D}}

\newcommand{\calR}{\mathcal{R}}

\newcommand{\calF}{\mathcal{F}}

\newcommand{\calP}{\mathcal{P}}
\newcommand{\calS}{\mathcal{S}}

\allowdisplaybreaks

\title{Optimal discrepancy rate of point sets \\ in Besov spaces with negative smoothness}
\author{Ralph Kritzinger \thanks{The author is supported by the Austrian Science Fund (FWF): Project F5509-N26, which is a part of the Special Research Program "Quasi-Monte Carlo Methods: Theory and Applications".}}
\date{}

\begin{document}

\setlength{\parindent}{0em} 

\maketitle

\begin{abstract}
We consider the local discrepancy of a symmetrized version of Hammersley type point sets in the unit square.
As a measure for the irregularity of distribution we study the norm of the local discrepancy
in Besov spaces with dominating mixed smoothness. It is known that for Hammersley type points
this norm has the best possible rate provided that the smoothness parameter of the Besov space is nonnegative.
While these point sets fail to achieve the same for negative smoothness, we will prove
in this note that the symmetrized versions overcome this defect. We conclude with some consequences on discrepancy in further function spaces with dominating mixed smoothness
and on numerical integration based on quasi-Monte Carlo rules.
\end{abstract} 

\centerline{\begin{minipage}[hc]{130mm}{
}
\end{minipage}}

 \allowdisplaybreaks
 
\section{Introduction}

For a multiset $\calP$ of $N\geq 1$ points in the unit square $[0,1]^2$ we define the local discrepancy as
$$ D_{\calP}(\bst):=\frac{1}{N}\sum_{\bsz \in \calP}\bsone_{[\bszero,\bst)}(\bsz)-t_1t_2. $$
Here $\bsone_{I}$ denotes the indicator function of an interval $I\subseteq [0,1)^2$. For $\bst=(t_1,t_2)\in [0,1]^2$ we set $[\bszero,\bst):=[0,t_1)\times [0,t_2)$ with volume $t_1t_2$.
To obtain a global measure for the irregularity of a point distribution $\calP$, one usually considers a norm of
the local discrepancy in some function space. A popular choice are the $L_p$ spaces for $p\in[1,\infty]$, which
are defined as the collection of all functions $f$ on $[0,1)^2$ with finite $L_p([0,1)^2)$ norm. For $p=\infty$ this norm is the supremum norm, i.e.
$$ \left\|f| L_{\infty}([0,1)^2)\right\|:=\sup_{\bst \in [0,1]^2}|f(\bst)|, $$
and for $p\in [1,\infty)$ these norms are given by $$\left\|f| L_p([0,1)^2)\right\|:=\left(\int_{[0,1)^2}|f(\bst)|^p\mathrm{d} \bst\right)^{\frac{1}{p}}.$$
Throughout this note, for functions $f,g:\NN \rightarrow \RR^+$, we write $g(N) \lesssim f(N)$ and $g(N) \gtrsim f(N)$, 
if there exists a constant $C>0$ independent of $N$ such that $g(N) \le C f(N)$ or $g(N) \geq C f(N)$ for all $N \in \NN$, $N \ge 2$, respectively. 
We write $f(N) \asymp g(N)$ to express that $g(N) \lesssim f(N)$ and $g(N) \gtrsim f(N)$ holds simultaneously. It is a well-known fact that for every $p\in[1,\infty]$ and $N\in\NN$ any $N$-element point set $\calP$ in $[0,1)^2$ satisfies 
\begin{equation}\label{lowproinovkritzbesov}
\left\|D_{\calP}| L_p([0,1)^2)\right\| \gtrsim N^{-1}(\log N)^{\frac{1}{2}}.
\end{equation} This inequality was shown by Roth~\cite{Roth} for $p=2$ (and therefore for $p \in (2,\infty]$ because of the monotonicity of the $L_p$ norms) and Schmidt~\cite{schX}  for $p\in (1,2)$. From the work of Hal\'{a}sz~\cite{hala} we know that it also holds for $p=1$. 
In recent years
several other norms of the local discrepancy have been studied. In this note we would like to investigate the discrepancy of certain point sets in
Besov spaces $S_{p,q}^rB([0,1)^2)$ with dominating mixed smoothness. The parameter $p$ describes the integrability of functions
belonging to this space, while $r$ is related to the smoothness of these functions. The third parameter $q$ is a regulation parameter. A definition of $S_{p,q}^rB([0,1)^2)$ can be found in Section~\ref{prelimkritzbesov}. We denote the Besov norm of a function $f$ by $\|f|S_{p,q}^rB([0,1)^2)\|$. 
The study of discrepancy in function spaces with dominating mixed smoothness was initiated by Triebel \cite{Tri10,Tri10b}, since
it is directly connected to numerical integration. He could show that for all $1\leq p,q \leq \infty$ and $r\in \RR$ satisfying $\frac{1}{p}-1<r<\frac{1}{p}$ and $q<\infty$ if $p=1$ and $q>1$ if $p=\infty$ and for any $N\in\NN$ the local discrepancy of any $N$-element point set $\calP$ in $[0,1)^2$ satisfies
\begin{equation} \label{lowerboundkritzbesov} \left\| D_{\calP} | S_{p,q}^rB([0,1)^2) \right\|\gtrsim N^{r-1}(\log{N})^{\frac{1}{q}}. \end{equation}
Also, for any $N\geq 2$, there exists a point set $\calP$ in $[0,1)^2$ with $N$ points such that
$$ \left\| D_{\calP} | S_{p,q}^rB([0,1)^2) \right\|\lesssim N^{r-1}(\log{N})^{\left(\frac{1}{q} +1-r\right)}. $$
Hinrichs showed in \cite{hin2010} that the gap between the exponents of the lower and the upper bounds can be closed for $1\leq p,q \leq \infty$ and $0 \leq r <\frac{1}{p}$ and that the lower bound \eqref{lowerboundkritzbesov} is sharp. He used Hammersley type point sets $\calR_n$ as introduced below.
It follows from his proof that these point sets can not be used to close the gap also for the parameter range $1/p-1<r<0$. It remained an open problem to find a point set which closes this gap also for this negative smoothness range. This problem was again mentioned in \cite[Problem 3]{hin2014} (here also for higher dimensions) and \cite[Remark 6.8]{Ull}. It is
the aim of this note to show that a solution is possible by applying some simple modifications to the point sets $\calR_n$, which will lead to the main result of this note.  \vspace{5mm}

In \cite{hin2010} Hinrichs studied the class of Hammersley type point sets
$$ \calR_n:=\left\{\left(\frac{t_n}{2}+\frac{t_{n-1}}{2^2}+\dots+\frac{t_1}{2^n},\frac{s_1}{2}+\frac{s_2}{2^2}+\dots+\frac{s_n}{2^n}\right)| t_1,\dots,t_n \in \{0,1\}\right\} $$
for $n\in\NN$, where $s_i=t_i$ or $s_i=1-t_i$ depending on $i$. It is obvious that $\calR_n$ has $2^n$ elements. We fix $\calR_n$ and introduce three
connected point sets by
\begin{eqnarray*}
    \calR_{n,1}&:=&\{(x,1-y)|(x,y)\in\calR_n\}, \\
		\calR_{n,2}&:=&\{(1-x,y)|(x,y)\in\calR_n\}, \\
		\calR_{n,3}&:=&\{(1-x,1-y)|(x,y)\in\calR_n\}. 
\end{eqnarray*} 
We set $\widetilde{\calR}_n:=\calR_n\cup \calR_{n,1} \cup \calR_{n,2} \cup \calR_{n,3}$ and call $\widetilde{\calR}_n$ a symmetrized Hammersley type point set. 
In literature one often finds a symmetrization in the sense of Davenport~\cite{daven}, which would be $\calR_n\cup \calR_{n,1}$. However, for our purposes we need
to work with the point set $\widetilde{\calR}_n$, which has $N=2^{n+2}$ elements, where some points might coincide.
With the point sets $\widetilde{\calR}_n$ we have the following main result of this note.

\begin{theorem} \label{theo1kritzbesov} Let $1\leq p,q \leq \infty$ and $r\in\RR$ such that $1/p-1<r<1/p$. Then the point sets $\widetilde{\calR}_n$ in $[0,1)^2$ with $N=2^{n+2}$ elements satisfy
$$ \|D_{\widetilde{\calR}_n}|S_{p,q}^rB([0,1)^2)\|\lesssim N^{r-1}(\log{N})^{1/q}.  $$
\end{theorem}
 
We would like to stress again that our result improves on \cite[Theorem 1.1]{hin2010} in the sense that we extended the range for the smoothness parameter $r$ to negative values.

\section{Preliminaries} \label{prelimkritzbesov}
We give a definition of the Besov spaces with dominating mixed smoothness.
Let therefore $\calS(\RR^2)$ denote the Schwartz space and $\calS'(\RR^2)$ the space of tempered distributions on $\RR^2$.
For $f \in \calS'(\RR^2)$ we denote by $\calF f$ the Fourier transform of $f$ and by $\calF^{-1} f$ its inverse. Let $\phi_0\in \calS(\RR)$ satisfy $\phi_0(t)=1$ for $|t|\leq 1$ and $\phi_0(t)=0$ for $|t|> \frac{3}{2}$. Let
$$ \phi_k(t)=\phi_0(2^{-k}t)-\phi_0(2^{-k+1}t), $$
where $t\in \RR, k\in \NN$, and $\phi_{\bsk}(\bst)=\phi_{k_1}(t_1) \phi_{k_2}(t_2)$
for $\bsk=(k_1,k_2)\in\NN_0^2$, $\bst=(t_1,t_2)\in \RR^2$. We note that $\sum_{\bsk \in \NN_0^2}\phi_{\bsk}(\bst)=1$ for all $\bst \in \RR^2$. The functions $\calF^{-1}(\phi_{\bsk}\calF f)$ are entire analytic functions for any $f \in \calS'(\RR^2)$. Let $0<p,q\leq \infty$ and $r\in \RR$. The Besov space $S_{p,q}^rB(\RR^2)$ of dominating mixed smoothness consists of all $f \in \calS'(\RR^2)$ with finite quasi-norm
$$ \left\|f| S_{p,q}^rB(\RR^2) \right\| =\left(\sum_{\bsk \in \NN_0^2}2^{r(k_1+k_2)q}\left\|\calF^{-1}(\phi_{\bsk} \calF f) | L_p(\RR^2)\right\|^q\right)^{\frac{1}{q}}, $$
with the usual modification if $q=\infty$.
Let $\calD([0,1)^2)$ be the set of all complex-valued infinitely differentiable functions on $\RR^2$ with compact support in the interior of $[0,1)^2$ and let $\calD'([0,1)^2)$ be its dual space of all distributions in $[0,1)^2$. The Besov space $S_{p,q}^rB([0,1)^2)$ of dominating mixed smoothness on the domain $[0,1)^2$ consists of all functions $f \in \calD'([0,1)^2)$ with finite quasi norm
$$ \left\|f| S_{p,q}^rB([0,1)^2)\right\|=\inf{\left\{\left\|g | S_{p,q}^rB(\RR^2)\right\|:g \in S_{p,q}^rB(\RR^2), g|_{[0,1)^2}=f\right\}}. $$

Actually, we will not make use of this technical definition. For our approach it is more convenient
to employ a characterization of Besov spaces via Haar functions, which we define in the following. 

A dyadic interval of length $2^{-j}, j\in {\mathbb N}_0,$ in $[0,1)$ is an interval of the form 
$$ I=I_{j,m}:=\left[\frac{m}{2^j},\frac{m+1}{2^j}\right) \ \ \mbox{for } \  m=0,1,\ldots,2^j-1.$$ We also define $I_{-1,0}=[0,1)$.
The left and right half of $I_{j,m}$ are the dyadic intervals $I_{j+1,2m}$ and $I_{j+1,2m+1}$, respectively. For $j\in\NN_0$, the Haar function $h_{j,m}$  
is the function on $[0,1)$ which is  $+1$ on the left half of $I_{j,m}$, $-1$ on the right half of $I_{j,m}$ and 0 outside of $I_{j,m}$. The $L_\infty$-normalized Haar system consists of
all Haar functions $h_{j,m}$ with $j\in{\mathbb N}_0$ and  $m=0,1,\ldots,2^j-1$ together with the indicator function $h_{-1,0}$ of $[0,1)$.
Normalized in $L_2([0,1))$ we obtain the orthonormal Haar basis of $L_2([0,1))$. 

Let ${\mathbb N}_{-1}=\NN_0 \cup \{-1\}$ and define ${\mathbb D}_j=\{0,1,\ldots,2^j-1\}$ for $j\in{\mathbb N}_0$ and ${\mathbb D}_{-1}=\{0\}$.
For $\bsj=(j_1,j_2)\in{\mathbb N}_{-1}^2$ and $\bsm=(m_1,m_2)\in {\mathbb D}_{\bsj} :={\mathbb D}_{j_1} \times {\mathbb D}_{j_2}$, 
the Haar function $h_{\bsj,\bsm}$ is given as the tensor product 
$$h_{\bsj,\bsm}(\bst) = h_{j_1,m_1}(t_1) h_{j_2,m_2}(t_2) \ \ \ \mbox{ for } \bst=(t_1,t_2)\in[0,1)^2.$$
We speak of $I_{\bsj,\bsm} = I_{j_1,m_1} \times I_{j_2,m_2}$ as dyadic boxes. 

We have the following crucial result \cite[Theorem 2.41]{Tri10}. 

\begin{proposition} \label{equivalencekritzbesov} 
Let $0<p,q\leq \infty$, $1<q\leq \infty$ if $p=\infty$, and $\frac{1}{p}-1<r<\min\left\{\frac{1}{p},1\right\}$. Let $f \in \calD'([0,1)^2)$.
Then $f \in S_{p,q}^rB([0,1)^2)$ if and only if it can be represented as
$$ f=\sum_{\bsj \in \NN_{-1}^2}\sum_{\bsm \in \bbD_{\bsj}}\mu_{\bsj, \bsm}2^{\max\{0,j_1\}+\max\{0,j_2\}}h_{\bsj,\bsm} $$
for some sequence $(\mu_{\bsj,\bsm})$ satisfying
\[  \left(\sum_{\bsj \in \NN_{-1}^2}2^{(j_1+j_2)\left(r-\frac{1}{p}+1\right)q}\left( \sum_{\bsm \in \bbD_{\bsj}}\left| \mu_{\bsj,\bsm}\right|^p\right)^{\frac{q}{p}}\right)^{\frac{1}{q}}<\infty, \]
where the convergence is unconditional in $\calD'([0,1)^2)$ and in any $S_{p,q}^{\rho}B([0,1)^2)$ with $\rho< r$. This representation of $f$ is unique with the  Haar coefficients 
$$ \mu_{\bsj,\bsm}=\mu_{\bsj,\bsm}(f)=\int_{[0,1)^2}f(\bst)h_{\bsj,\bsm}(\bst)\rd \bst. $$
The expression on the left-hand-side of the above inequality provides an equivalent quasi-norm on $S_{p,q}^rB([0,1)^2)$, i.e.
\begin{equation*} \left\|f| S_{p,q}^rB([0,1)^2) \right\| \asymp  \left(\sum_{\bsj \in \NN_{-1}^2}2^{(j_1+j_2)\left(r-\frac{1}{p}+1\right)q}\left( \sum_{\bsm \in \bbD_{\bsj}}\left| \mu_{\bsj,\bsm}\right|^p\right)^{\frac{q}{p}}\right)^{\frac{1}{q}}. \end{equation*}
\end{proposition}
 We will follow the same approach as Hinrichs and first estimate the Haar coefficients of $D_{\widetilde{\calR}_n}$ and then apply Proposition~\ref{equivalencekritzbesov}. This note is therefore similar in structure to \cite{hin2010} and uses several results from there.

\section{Proof of Theorem~\ref{theo1kritzbesov}}

To begin with, we state several auxiliary results from \cite[Lemmas 3.2--3.4, 3.6]{hin2010}.
\begin{lemma} \label{volumekritzbesov} Let $f(\bst)=t_1t_2$ for $\bst=(t_1,t_2)\in [0,1)^2$. For $\bsj\in\NN_{-1}^2$ and $\bsm\in\bbD_{\bsj}$ let $\mu_{\bsj,\bsm}$ be the Haar coefficients of $f$. Then
   \begin{itemize}
	     \item[(i)] If $\bsj=(j_1,j_2)\in\NN_0^2$ then $\mu_{\bsj,\bsm}=2^{-2(j_1+j_2+2)}$.
			 \item[(ii)] \hspace{0.5mm} If $\bsj=(-1,k)$ or $\bsj=(k,-1)$ with $k\in\NN_0$ then $\mu_{\bsj,\bsm}=-2^{-(2k+3)}$.
	 \end{itemize}
\end{lemma}

\begin{lemma} \label{countingkritzbesov} Fix $\bsz=(z_1,z_2)\in[0,1)^2$ and let $f(\bst)=\bsone_{[\bszero,\bst)}(\bsz)$ for $\bst=(t_1,t_2)\in [0,1)^2$. For $\bsj\in\NN_{-1}^2$ and $\bsm=(m_1,m_2)\in\bbD_{\bsj}$ let $\mu_{\bsj,\bsm}$ be the Haar coefficients of $f$. Then $\mu_{\bsj,\bsm}=0$ whenever
$\bsz \notin \mathring{I}_{\bsj,\bsm}$, where $\mathring{I}_{\bsj,\bsm}$ denotes the interior of $I_{\bsj,\bsm}$. If $\bsz \in \mathring{I}_{\bsj,\bsm}$ then
   \begin{itemize}
	     \item[(i)] If $\bsj=(j_1,j_2)\in\NN_0^2$ then $$\mu_{\bsj,\bsm}=2^{-(j_1+j_2+2)}(1-|2m_1+1-2^{j_1+1}z_1|)(1-|2m_2+1-2^{j_2+1}z_2|).$$
			 \item[(ii)] \hspace{0.5mm} If $\bsj=(-1,k)$, $k\in\NN_0$, then $\mu_{\bsj,\bsm}=-2^{-(k+1)}(1-z_1)(1-|2m_2+1-2^{k+1}z_2|)$.
			 \item[(iii)] \hspace{1mm} If $\bsj=(k,-1)$, $k\in\NN_0$, then $\mu_{\bsj,\bsm}=-2^{-(k+1)}(1-z_2)(1-|2m_1+1-2^{k+1}z_1|)$.
	 \end{itemize}
\end{lemma}

\begin{lemma} \label{formulaskritzbesov} Let $\calR_n$ be a Hammersley type point set with $2^n$ points. Let $\bsj=(j_1,j_2)\in\NN_0^2$ and $\bsm=(m_1,m_2)\in\bbD_{\bsj}$.
 Then, if $j_1+j_2<n$,
  $$ \sum_{\bsz \in \calR_n \cap \mathring{I}_{\bsj,\bsm}}(1-|2m_1+1-2^{j_1+1}z_1|)=\sum_{\bsz \in \calR_n \cap \mathring{I}_{\bsj,\bsm}}(1-|2m_2+1-2^{j_2+1}z_2|)=2^{n-j_1-j_2-1} $$
	and, if $j_1+j_2<n-1$,
	$$  \sum_{\bsz \in \calR_n \cap \mathring{I}_{\bsj,\bsm}}(1-|2m_1+1-2^{j_1+1}z_1|)(1-|2m_2+1-2^{j_2+1}z_2|)=2^{n-j_1-j_2-2}+2^{j_1+j_2-n}. $$
\end{lemma}

Now we are ready to compute the Haar coefficients of $D_{\widetilde{\calR}_n}$.

\begin{proposition}  \label{Haarckritzbesov}
Let $\widetilde{\calR}_n$ be a symmetrized Hammersley type point set with $N=2^{n+2}$ elements and let $f$ be the local discrepancy
of $\widetilde{\calR}_n$ and $\mu_{\bsj,\bsm}$ the Haar coefficients of $f$ for $\bsj\in\NN_{-1}^2$ and $\bsm=(m_1,m_2)\in\bbD_{\bsj}$.

Let $\bsj=(j_1,j_2)\in \NN_{0}^2$. Then
 \begin{itemize}
  \item[(i)] if $j_1+j_2<n-1$ and $j_1,j_2\ge 0$ then $|\mu_{\bsj,\bsm}| = 2^{-2(n+1)}$.
  \item[(ii)] \hspace{0.5mm} if $j_1+j_2\ge n-1$ and $0\le j_1,j_2\le n$ then $|\mu_{\bsj,\bsm}| \le 2^{-(n+j_1+j_2)}$ and
     $|\mu_{\bsj,\bsm}| = 2^{-2(j_1+j_2+2)}$ for all but at most $2^{n+2}$ coefficients $\mu_{\bsj,\bsm}$ with $\bsm\in {\mathbb D}_{\bsj}$.
  \item[(iii)] \hspace{1mm} if $j_1 \ge n$ or $j_2 \ge n$ then $|\mu_{\bsj,\bsm}| = 2^{-2(j_1+j_2+2)}$.
 \end{itemize} 
 
 Now let $\bsj=(-1,k)$ or $\bsj=(k,-1)$ with $k\in \NN_0$. Then
 \begin{itemize}
  \item[(iv)] \hspace{0.5mm} if $k<n$ then $\mu_{\bsj,\bsm}=0$.
  \item[(v)] if $k\ge n$  then $|\mu_{\bsj,\bsm}| = 2^{-(2k+3)}$.
 \end{itemize}
 
 Finally, 
  \begin{itemize}
   \item[(vi)] \hspace{0.5mm} $\mu_{(-1,-1),(0,0)} = 0$.
  \end{itemize} 
\end{proposition}

\begin{proof} The cases $(iii)$ and $(v)$ follow from the fact that no elements of $\widetilde{\calR}_n$ are contained in the interior of a dyadic box $I_{(j_1,j_2),\bsm}$ if $j_1\geq n$ or $j_2\geq n$, together with Lemma~\ref{volumekritzbesov}. We consider the case $(ii)$. For a fixed $\bsj=(j_1,j_2)$ the interiors of the dyadic boxes $I_{\bsj,\bsm}$ for $\bsm\in\bbD_{\bsj}$ are mutually
disjunct and at most $2^{n+2}$ of these boxes can contain points from $\widetilde{\calR}_n$. We have $\mu_{\bsj,\bsm}=2^{-2(j_1+j_2+2)}$ if the corresponding box $I_{\bsj,\bsm}$ is empty. The other boxes contain at most $8$ points (because the volume of $I_{\bsj,\bsm}$ is at most $2^{-(n-1)}$ due to the condition $j_1+j_2\geq n-1$ and because of the net property of $\calR_n$ and its connected point sets). Together with the first part of Lemma~\ref{countingkritzbesov} and the triangle inequality this
yields $|\mu_{\bsj,\bsm}|\leq 8\cdot 2^{-(n+2)}2^{-(j_1+j_2+2)}+2^{-2(j_1+j_2+2)} \leq 2^{-(n+j_1+j_2)}$.

The case $(vi)$ can be seen as follows:
\begin{align*}
   \mu_{(-1,-1),(0,0)}=&\int_{0}^{1}\int_{0}^{1}D_{\widetilde{\calR}_n}(t_1,t_2)\rd t_1\rd t_2 = \frac{1}{N}\sum_{\bsz \in \widetilde{\calR}_n} \int_{z_1}^{1}\int_{z_2}^{1}1\rd t_1\rd t_2-\int_{0}^{1}\int_{0}^{1}t_1t_2\rd t_1\rd t_2 \\
	=&\frac{1}{2^{n+2}}\sum_{\bsz \in \widetilde{\calR}_n}(1-z_1)(1-z_2)-\frac{1}{4} \\
	=&\frac{1}{2^{n+2}}\sum_{(x,y) \in \calR_n}[(1-x)(1-y)+(1-x)y+x(1-y)+xy]-\frac{1}{4}\\
	=&\frac{1}{2^{n+2}}\sum_{(x,y) \in \calR_n}1-\frac{1}{4}=\frac{1}{2^{n+2}}2^n-\frac14=0.
\end{align*}
To show the claim in $(iv)$ for the case $\bsj=(k,-1)$ with $k\in\NN_0$, $k<n$, we have to consider the expression
  $$ S:=\sum_{\bsz \in \widetilde{\calR}_n \cap \mathring{I}_{(k,-1),(m_1,0)}}(1-|2m_1+1-2^{k+1}z_1|)(1-z_2) $$
	for any $m_1\in\{0,\dots,2^k-1\}$. We can write
	\begin{align*}
	    S=&\sum_{(x,y) \in \calR_n \cap \mathring{I}_{(k,-1),(m_1,0)}}(1-|2m_1+1-2^{k+1}x|)(1-y) \\
	      &+\sum_{(x,1-y) \in \calR_n \cap \mathring{I}_{(k,-1),(m_1,0)}}(1-|2m_1+1-2^{k+1}x|)y \\
				&+\sum_{(1-x,y) \in \calR_n \cap \mathring{I}_{(k,-1),(m_1,0)}}(1-|2m_1+1-2^{k+1}(1-x)|)(1-y) \\
				&+\sum_{(1-x,1-y) \in \calR_n \cap \mathring{I}_{(k,-1),(m_1,0)}}(1-|2m_1+1-2^{k+1}(1-x)|)y \\
				=&\sum_{(x,y) \in \calR_n \cap \mathring{I}_{(k,-1),(m_1,0)}}(1-|2m_1+1-2^{k+1}x|) \\
				&+\sum_{(1-x,y) \in \calR_n \cap \mathring{I}_{(k,-1),(m_1,0)}}(1-|2m_1+1-2^{k+1}(1-x)|)=:S_1+S_2,
	\end{align*}
	where we used the obvious equivalences $(x,y) \in \calR_n \cap \mathring{I}_{(k,-1),(m_1,0)}$ if and only if $(x,1-y) \in \calR_n \cap \mathring{I}_{(k,-1),(m_1,0)}$ as well as $(1-x,y) \in \calR_n \cap \mathring{I}_{(k,-1),(m_1,0)}$ if and only if $(1-x,1-y) \in \calR_n \cap \mathring{I}_{(k,-1),(m_1,0)}$ in the last step.
	Since the interval $\mathring{I}_{(k,-1),(m_1,0)}$ is the same as $\mathring{I}_{(k,0),(m_1,0)}$, we obtain 
	$S_1=2^{n-k-1}$ from the first part of Lemma~\ref{formulaskritzbesov}. To evaluate $S_2$ we observe
	that 
	$$ 1-x \in \mathring{I}_{k,m_1} \Leftrightarrow \frac{m_1}{2^k} < 1-x < \frac{m_1+1}{2^k} 
	\Leftrightarrow \frac{2^k-1-m_1}{2^k} < x < \frac{2^k-m_1}{2^k} \Leftrightarrow x \in \mathring{I}_{k,\widetilde{m}_1}, $$
	where we set $\widetilde{m}_1=2^k-1-m_1$. This yields the equivalence of $(1-x,y) \in \calR_n \cap \mathring{I}_{(k,-1),(m_1,0)}$ and $(x,y) \in \calR_n \cap \mathring{I}_{(k,-1),(\widetilde{m}_1,0)}$. We also find
	\begin{align*} |2m_1+1-2^{k+1}(1-x)|=&|2(m_1+1-2^k)-1+2^{k+1}x|\\ =&|-2\widetilde{m}_1-1+2^{k+1}x|=|2\widetilde{m}_1+1-2^{k+1}x|
	\end{align*}
	and hence we obtain
	\begin{align*} S_2=&\sum_{(x,y) \in \calR_n \cap \mathring{I}_{(k,-1),(\widetilde{m}_1,0)}}(1-|2\widetilde{m}_1+1-2^{k+1}x|)=2^{n-k-1},\end{align*}
	where we regarded the first part of Lemma~\ref{formulaskritzbesov} again. Altogether, we have
	\begin{align*}
	   \mu_{(k,-1),(m_1,0)}=&-\frac{1}{N}2^{-(k+1)}(S_1+S_2)-(-2^{-(2k+3)}) \\
		                     =&-2^{-(n+2)}2^{-(k+1)}2^{n-k}+2^{-(2k+3)}=0
	\end{align*}
	with Lemmas~\ref{volumekritzbesov} and~\ref{countingkritzbesov},
	and this part of the proposition is verified. It is clear that the result for $\mu_{(-1,k),(0,m_2)}$ if $k<n$ can be shown analogously.
	
	Finally, we prove $(i)$ and therefore have to analyze the sum
$$ T:=\sum_{\bsz\in\widetilde{\calR}_n \cap \mathring{I}_{\bsj,\bsm}}(1-|2m_1+1-2^{j_1+1}z_1|)(1-|2m_2+1-2^{j_2+1}z_2|),$$
where $\bsj=(j_1,j_2)\in\NN_{0}^2$ with $j_1+j_2<n-1$. We have
\begin{align*}
   T=&\sum_{(x,y)\in\calR_n \cap \mathring{I}_{\bsj,\bsm}}(1-|2m_1+1-2^{j_1+1}x|)(1-|2m_2+1-2^{j_2+1}y|) \\
	  &+\sum_{(x,1-y)\in\calR_n \cap \mathring{I}_{\bsj,\bsm}}(1-|2m_1+1-2^{j_1+1}x|)(1-|2m_2+1-2^{j_2+1}(1-y)|) \\
		&+\sum_{(1-x,y)\in\calR_n \cap \mathring{I}_{\bsj,\bsm}}(1-|2m_1+1-2^{j_1+1}(1-x)|)(1-|2m_2+1-2^{j_2+1}y|) \\
		&+\sum_{(1-x,1-y)\in\calR_n \cap \mathring{I}_{\bsj,\bsm}}(1-|2m_1+1-2^{j_1+1}(1-x)|)(1-|2m_2+1-2^{j_2+1}(1-y)|) \\
	  =:&T_1+T_2+T_3+T_4. 
\end{align*}
We obtain directly from the second part of Lemma~\ref{formulaskritzbesov} that $T_1=2^{n-j_1-j_2-2}+2^{j_1+j_2-n}$.
With the same arguments as in the proof of $(iv)$ we can show
\begin{eqnarray*}
T_2&=&\sum_{(x,y)\in\calR_n \cap \mathring{I}_{\bsj,(m_1,\widetilde{m}_2)}}(1-|2m_1+1-2^{j_1+1}x|)(1-|2\widetilde{m}_2+1-2^{j_2+1}y|), \\
T_3&=&\sum_{(x,y)\in\calR_n \cap \mathring{I}_{\bsj,(\widetilde{m}_1,m_2)}}(1-|2\widetilde{m}_1+1-2^{j_1+1}x|)(1-|2m_2+1-2^{j_2+1}y|), \\
T_4&=&\sum_{(x,y)\in\calR_n \cap \mathring{I}_{\bsj,(\widetilde{m}_1,\widetilde{m}_2)}}(1-|2\widetilde{m}_1+1-2^{j_1+1}x|)(1-|2\widetilde{m}_2+1-2^{j_2+1}y|),
\end{eqnarray*}
where $\widetilde{m}_i=2^{j_i}-1-m_i$ for $i\in\{1,2\}$. But from this and Lemma~\ref{formulaskritzbesov} we see that $T_2=T_3=T_4=T_1$ and
together with Lemma~\ref{volumekritzbesov} and Lemma~\ref{countingkritzbesov}
\begin{align*}
   \mu_{\bsj,\bsm}=&\frac{1}{N}2^{-(j_1+j_2+2)}(T_1+T_2+T_3+T_4)-2^{-2(j_1+j_2+2)} \\
	                =&2^{-(n+2)}2^{-(j_1+j_2+2)}(2^{n-j_1-j_2}+2^{j_1+j_2-n+2})-2^{-2(j_1+j_2+2)}=2^{-2(n+1)}
\end{align*}
as claimed. The proof of the proposition is complete.
\end{proof}

\vspace{5mm}

Now we are able to prove Theorem~\ref{theo1kritzbesov}.

\vspace{3mm}

\begin{proof} We consider any symmetrized Hammersley type point set $\widetilde{\calR}_n$ (we do not have to specify the dependence of the
$s_i$ on $t_i$ in the definition of $\calR_n$). For $\bsj\in\NN_{-1}^2$ and $\bsm\in\bbD_{\bsj}$ let $\mu_{\bsj,\bsm}$ be the Haar coefficients of the local discrepancy of $\widetilde{\calR}_n$. According to Proposition~\ref{equivalencekritzbesov}, it suffices to show that for all $p,q,r$ satisfying the conditions in Theorem~\ref{theo1kritzbesov} we have
	\begin{equation} \label{sufficekritzbesov} \left(\sum_{\bsj \in \NN_{-1}^2}2^{(j_1+j_2)\left(r-\frac{1}{p}+1\right)q}\left( \sum_{\bsm \in \bbD_{\bsj}}\left| \mu_{\bsj,\bsm}\right|^p\right)^{\frac{q}{p}}\right)^{\frac{1}{q}}\lesssim 2^{n(r-1)}n^{1/q}. \end{equation}
	This yields
	$$ \|D_{\widetilde{\calR}_n}|S_{p,q}^rB([0,1)^2)\|\lesssim 2^{-2(r-1)}2^{(n+2)(r-1)}(n+2)^{1/q}\lesssim N^{r-1}(\log{N})^{1/q}. $$
	 To verify \eqref{sufficekritzbesov}, we split the sum over $\bsj$ in six cases
	according to Proposition~\ref{Haarckritzbesov} (and thereby applying Minkowski's inequality). We remark that the cases $(i)$, $(ii)$, $(iii)$ and $(v)$ have already been treated
	in \cite[Section 4]{hin2010}, since in these cases the bounds on the Haar coefficients of $D_{\calR_n}$ are (basically) the same
	as those for the Haar coefficients of $D_{\widetilde{\calR}_n}$. In all cases Hinrichs obtained an upper bound of the form $c2^{n(r-1)}n^{1/q}$ with $c$ independent of $n$ for the whole parameter range $1/p-1<r<1/p$. The only cases where the condition $r\geq 0$ was necessary were $(iv)$ and $(vi)$. However, the symmetrization of $\calR_n$ has the effect that the corresponding Haar coefficients of $D_{\widetilde{\calR}_n}$ vanish in these two cases, and the result follows. 
\end{proof}

\begin{remark} Let $f$ be the local discrepancy of the point set $\calR_n\cup \calR_{n,1}$ and $\mu_{\bsj,\bsm}$ for $\bsj\in\NN_{-1}^2$ and $\bsm\in\bbD_{\bsj}$ be the
corresponding Haar coefficients. Then one can show that $\mu_{(-1,-1),(0,0)}=2^{-(n+2)}$ and $\mu_{(-1,k),(0,m_2)}=-2^{-(n+2k+3)}+2^{-(2n+2)}T_k$ for $k\in\NN_0$, $k<n$. Here, $T_k=1$ if $s_{k+1}=t_{k+1}$ and $T_k=-1$ if $s_{k+1}=1-t_{k+1}$ in the definition of $\calR_n$. Hence, the proof of Theorem~\ref{theo1kritzbesov} does not work for this class of point sets.
\end{remark}

\section{Discrepancy in further function spaces and numerical integration}

As pointed out in \cite{Mar2013,Mar2013b,Tri10} one can easily deduce results on the discrepancy of point sets
in Triebel-Lizorkin spaces from the discrepancy estimates in Besov spaces. Let $0<p< \infty$, $0<q\leq \infty$ and $r\in\RR$. The Triebel-Lizorkin space $S_{p,q}^rF(\RR^2)$ with dominating mixed smoothness consists of all $f\in \calS'(\RR^2)$ with finite quasi-norm
$$ \|f|S_{p,q}^rF(\RR^2)\|=\left\|\left(\sum_{\bsk\in\NN_0^2}2^{r(k_1+k_2)q}|\calF^{-1}(\Phi_{\bsk}\calF f)(\cdot)|^q\right)^{1/q}|L_p(\RR^2)\right\| $$
with the usual modification if $q=\infty$. The space $S_{p,q}^rF([0,1)^2)$ can be introduced analogously to $S_{p,q}^rB([0,1)^2)$.
For $0<p,q<\infty$ and $r\in\RR$ we have the embeddings
\begin{equation} \label{embeddingskritzbesov} S_{\max\{p,q\},q}^rB([0,1)^2)\hookrightarrow S_{p,q}^rF([0,1)^2) \hookrightarrow S_{\min\{p,q\},q}^rB([0,1)^2), \end{equation}
which were proven in \cite[Corollary 1.13]{Mar2013b}, based on other embedding theorems from \cite[Remark 6.28]{Tri10} and \cite[Proposition 2.3.7]{han}. From the first embedding together with Theorem~\ref{theo1kritzbesov} we obtain
\begin{corollary} \label{lizorkinkritzbesov} Let $1\leq p,q < \infty$ and $\frac{1}{\max\{p,q\}}-1<r<\frac{1}{\max\{p,q\}}$. Then the point sets $\widetilde{\calR}_n$ in $[0,1)^2$ with $N=2^{n+2}$ elements satisfy
$$ \|D_{\calP}|S_{p,q}^rF([0,1)^2)\|\lesssim N^{r-1}(\log{N})^{1/q}.  $$
\end{corollary}
This corollary improves on \cite[Theorem 6.1]{Mar2013}, where Hammersley type point sets in arbitrary base $b\geq 2$ have been considered, by extending again the range of $r$ to negative values. There exist corresponding lower bounds for the norm of the local discrepancy in Triebel-Lizorkin spaces for $\frac{1}{\min\{p,q\}}-1<r<\frac{1}{p}$ as shown in \cite[Corollary 4.2]{Mar2013b}. This follows from the lower bounds on the discrepancy in Besov spaces as stated in the introduction, together with the second embedding in \eqref{embeddingskritzbesov}. \vspace{5mm}

For $1<p<\infty$ the spaces $S_p^rH([0,1)^2):=S_{p,2}^rF([0,1)^2)$ are called Sobolev spaces with dominating mixed smoothness. By choosing $q=2$ in Corollary~\ref{lizorkinkritzbesov} we obtain an analogous result on Sobolev spaces. Further, it is well known that $S_{p}^0H([0,1)^2)=L_p([0,1)^2)$. Regarding this fact we derive from Corollary~\ref{lizorkinkritzbesov} that the symmetrized Hammersley type point sets achieve an $L_p$ discrepancy of order $N^{-1}(\log{N})^{1/2}$ for all $p\in[1,\infty)$, which is best possible in the sense of~\eqref{lowproinovkritzbesov}. This however is not so surprising, since in \cite[Theorem 3]{HKP14} it has been shown that already a Davenport type symmetrization of $\calR_n$ achieves the best possible rate of $L_p$ discrepancy for all $p\in[1,\infty)$, i.e. $\|D_{\calR_n\cup\calR_{n,1}}|L_p([0,1)^2)\|\lesssim N^{-1}(\log{N})^{1/2}$.  By different means as used in this note, a certain type of symmetrized Hammersley point sets with the optimal order of $L_p$ discrepancy in a prime base $b$ has been studied by Goda~\cite[Theorem 24]{goda}, which matches our construction of $\widetilde{\calR}_n$ for $b=2$. We observe that the construction of point sets with the optimal rate of discrepancy in Besov, Triebel-Lizorkin or Sobolev spaces with negative smoothness is even more subtle than to find point sets with the optimal order of $L_p$ discrepancy. \vspace{5mm}

Finally, we would like to add a few words concerning errors of quasi-Monte Carlo (QMC) methods for numerical integration in spaces with dominating mixed smoothness. For a function $f$ in a normed space $F$ of functions on $[0,1)^2$ we would like to approximate the integral $I(f):=\int_{[0,1)^2}f(\bsx)\rd\bsx$ by a QMC algorithm $Q_N(\calP,f)=\frac{1}{N}\sum_{i=1}^{N}f(\bsx_i)$, where $\calP=\{\bsx_1,\dots,\bsx_N\}$ is a set of $N$ points in the unit square. The minimal worst-case error of QMC algorithms with respect to a class of functions $F$ is defined as
$$ \text{err}_N(F):=\inf_{\#\calP=N}\sup_{\|f|F\|\leq1}|I(f)-Q_N(\calP,f)|.$$
The infimum is extended over all point sets in $[0,1)^2$ with $N$ elements and the supremum is extended over all functions in the unit ball of $F$. We state a remarkable connection between discrepancy and integration errors in Besov spaces. Let therefore
$$ \text{disc}_N(S_{p,q}^rB([0,1)^2)):=\inf_{\#\calP=N}\|D_{\calP}(\cdot)|S_{p,q}^rB([0,1)^2)\|. $$
It is known that $S_{p',q'}^{1-r}B([0,1)^2)^{\urcorner}$ with $1/p+1/p'=1/q+1/q'=1$ is the dual space of $S_{p,q}^rB([0,1)^2)$, where $S_{p',q'}^{1-r}B([0,1)^2)^{\urcorner}$ is the class of all functions in $S_{p',q'}^{1-r}B([0,1)^2)$ with zero boundary on the upper and right boundary line. Let $1\leq p,q \leq \infty$ ($q<\infty$ if $p=1$ and $q>1$ if $p=\infty$) and $1/p<r<1/p+1$. Then we have for every integer $N\geq 2$  
\begin{equation} \label{errorkritzbesov}  \text{err}_N(S_{p,q}^rB([0,1)^2)^{\urcorner}) \asymp \text{disc}_N(S_{p',q'}^{1-r}B([0,1)^2)), \end{equation}
which follows from \cite[Theorem 6.11]{Tri10}. This relation leads to the following result:
\begin{theorem} \label{theo2kritzbesov} Let $1\leq p,q \leq \infty$ ($q<\infty$ if $p=1$ and $q>1$ if $p=\infty$) and $1/p<r<1+1/p$. Then for $N=2^{n+2}$ with $n\in\NN$ we have
$$ \mathrm{err}_N(S_{p,q}^rB([0,1)^2)^{\urcorner})\lesssim N^{-r} (\log{N})^{1-1/q}. $$
\end{theorem}
\begin{proof} From \eqref{errorkritzbesov} we have
$$ \text{err}_N(S_{p,q}^rB([0,1)^2)^{\urcorner})\lesssim \text{disc}_N(S_{p',q'}^{1-r}B([0,1)^2)) $$
for $1/p<r<1+1/p$. Theorem~\ref{theo1kritzbesov} yields further
$$ \text{disc}_N(S_{p',q'}^{1-r}B([0,1)^2)) \lesssim N^{1-r-1}(\log{N})^{1/q'}=N^{-r}(\log{N})^{1-1/q} $$
for $1/p'-1<1-r<1/p'$. The last condition on $r$ is equivalent to $1/p<r<1+1/p$ and the result follows.
\end{proof}
We remark that there exists a corresponding lower bound on $\text{err}_N(S_{p,q}^rB([0,1)^2)$ which shows that the rate
of convergence in this theorem is optimal.
The novelty of Theorem~\ref{theo2kritzbesov} is the fact that in the two-dimensional case for $1<r<1+1/p$ the optimal rate of convergence can be achieved with QMC rules (based on symmetrized Hammersley type point sets). Previously, this has only been shown for the smaller parameter range $1/p<r<1$ in \cite[Theorem 5.6]{Mar2013b} (but for arbitrary dimensions). The smoothness range, for which the optimal order for the worst-case integration error is achieved, can be further extended if one either considers one-periodic functions only (see~\cite{Ull} for the case $s=2$ and \cite{HMOU} for a generalization to higher dimensions) or if one allows more general cubature rules that are not necessarily of QMC type. Results in this directions can be found for instance in~\cite{Ull}, where Hammersley type point sets were used as integration nodes of non-QMC rules, and~\cite{Ull2}, where Frolov lattices were proven to yield optimal convergence rates also for higher dimensions and for all $r>1/p$.\vspace{5mm}

 With similar arguments as above we obtain an analogous result on integration errors in Triebel-Lizorkin spaces (and hence in Sobolov spaces).
\begin{corollary} Let $1\leq p,q \leq \infty$ and $1/\min\{p,q\}<r<1+1/\min\{p,q\}$. Then for $N=2^{n+2}$ with $n\in\NN$ we have
$$ \mathrm{err}_N(S_{p,q}^rF([0,1)^2)^{\urcorner})\lesssim N^{-r} (\log{N})^{1-1/q}. $$
\end{corollary}
\begin{proof}
This result is a consequence of the second embedding in \eqref{embeddingskritzbesov}, which implicates 
$$\mathrm{err}_N(S_{p,q}^rF([0,1)^2)^{\urcorner})\leq\mathrm{err}_N(S_{\min\{p,q\},q}^rB([0,1)^2)^{\urcorner}),$$
 and Theorem~\ref{theo2kritzbesov}.
\end{proof}

\noindent{\bf Author's Address:}

\noindent Ralph Kritzinger, Institut f\"{u}r Finanzmathematik und angewandte Zahlentheorie, Johannes Kepler Universit\"{a}t Linz, Altenbergerstra{\ss}e 69, A-4040 Linz, Austria. Email: ralph.kritzinger(at)jku.at

\end{document}